\documentclass[final]{agujournal2019}
\usepackage{url} %

\usepackage{amsmath,amssymb, amsfonts,amsthm}
\usepackage{enumitem}
\usepackage{mathtools}
\usepackage{multirow,multicol}
\usepackage[ruled, linesnumbered ]{algorithm2e} %
\usepackage{xcolor}
\usepackage{ragged2e} %
\usepackage{nicematrix}
\usepackage{float}

\newcommand{\bfe}{\mathbf e}

\newcommand{\bfK}{\mathbf K}

\newcommand{\bfX}{\mathbf X}
\newcommand{\bfx}{\mathbf x}

\newcommand{\bfy}{\mathbf y}

\newcommand{\bfmu}{\boldsymbol{\mu}}

\newcommand{\bfSigma}{\boldsymbol{\Sigma}}

\newcommand{\calA}{\mathcal A}
\newcommand{\calD}{\mathcal D}

\newcommand{\calN}{\mathcal N}

\newcommand{\Var}{\mathrm{Var}}
\newcommand{\Cov}{\mathrm{Cov}}

\newcommand{\bbE}{\mathbb E}
\newcommand{\bbR}{\mathbb R}

\newcommand{\hatA}{\hat A}

\newcommand{\hatz}{\hat z}

\theoremstyle{plain}
\newtheorem{theorem}{Theorem}[section]

\newtheorem{lemma}[theorem]{Lemma}

\theoremstyle{definition}

\theoremstyle{remark}

\begin{document}
\justifying %

\title{Improving the Predictability of the Madden-Julian Oscillation at Subseasonal Scales with Gaussian Process Models}

\authors{Haoyuan Chen\affil{1}, Emil Constantinescu\affil{2}, Vishwas Rao\affil{2}, Cristiana Stan\affil{3}.}

\affiliation{1}{Department of Industrial and Systems Engineering, Texas A\&M University,\\ College Station, TX, USA }
\vspace{1mm}
\affiliation{2}{Mathematics and Computer Science Division,
  Argonne National Laboratory,\\
  Lemont, IL, USA}
\vspace{1mm}
\affiliation{3}{Department of Atmospheric, Oceanic and Earth Sciences, George Mason University,\\
  Fairfax, VA, USA}

\correspondingauthor{Emil Constantinescu}{emconsta@anl.gov}

\begin{keypoints}
\item Propose a probabilistic framework for MJO prediction using Gaussian process models and empirical correlations
\item Nonparametric model has a better prediction skill than ANN for the 5 forecast lead days in terms of correlation and overall in terms of RMSE 
\item The Gaussian process model provides the confidence intervals for the forecast at subseasonal scales (3 weeks) on average 
\end{keypoints}

\begin{abstract}
The Madden--Julian Oscillation (MJO) is an influential climate phenomenon that plays a vital role in modulating global weather patterns. In spite of the improvement in MJO predictions made by machine learning algorithms, such as neural networks, most of them cannot provide the uncertainty levels in the MJO forecasts directly. To address this problem, we develop a nonparametric strategy based on Gaussian process (GP) models. We calibrate GPs using empirical correlations and we propose a posteriori covariance correction. Numerical experiments demonstrate that our model has better prediction skills than the ANN models for the first five lead days. Additionally, our posteriori covariance correction extends the probabilistic coverage by more than three weeks.
\end{abstract}

\section*{Plain Language Summary}
The Madden--Julian Oscillation, or MJO, is a significant weather pattern that affects weather, influencing rainfall, temperature, and even storm frequency and intensity. When the MJO is active, it can affect the weather globally. 
To better predict  weather changes with 3-4 weeks in advance , we rely on the ability to predict the MJO's activity. Data-driven methods such as the ones that rely on deep neural networks have been recently employed to make such predictions. By examining existing MJO patterns, neural networks attempt to predict upcoming ones. However, while neural networks are robust enough to predict the MJO's activity, they do not provide confidence intervals for those predictions. To address this shortcoming, we use a model known as the ``Gaussian process" or GP. 
This statistical tool is distinctive because it not only provides predictions but also quantifies the level of confidence in them.

\section{Introduction}
\label{sec:intro}
The Madden--Julian Oscillation (MJO) \cite%
{madden1971detection,madden1972description} is the dominant mode of intraseasonal variability of the tropics \cite{zhang2013madden}. In the tropics, the MJO exerts its influence on weather and modulates  cyclone activity \cite{maloney2000modulation,camargo2009diagnosis} and El Nino Southern Oscillation \cite<ENSO;>{bergman2001intraseasonal,lybarger2019revisiting}. The MJO influence extends outside of the tropics and is one of the important sources of potential predictability on the subseasonal-to-seasonal (S2S) time scales in the extratropics \cite{stan2017review}. Originating in the equatorial Indian Ocean, the MJO propagates eastward along the equator alternating between phases of active and suppressed convection. Traditionally, the amplitude and phase of the MJO have been described by using various MJO indices derived from outgoing longwave radiation (OLR) alone (e.g., OLR MJO Index, OMI; real-time OLR MJO index, ROMI) or in combination with the zonal wind at 850 hPa and 200 hPa (Real-time Multivariate MJO, RMM). The RMM index \cite{wheeler2004all} consists of a pair (in quadrature) of principal component (PC) time series known as RMM1 and RMM2 ($\text{RMM}=\sqrt{\text{RMM1}^2+\text{RMM2}^2}$). RMM1 and RMM2 are the first two PCs of combined OLR and zonal winds in the lower (850 hPa) and upper (200 hPa) troposphere averaged between 15S and 15N.

Despite the MJO’s pivotal role in the climate system, significant gaps remain in our understanding of its underlying mechanisms. Consequently, climate models struggle to accurately reproduce the observed characteristics of the MJO \cite{chen2022mjo}, and forecast systems face limitations in predicting the MJO with skill beyond a two-week lead time \cite{kim2018prediction,lim2018mjo,kim2019mjo}.

 Recent advancements in machine learning (ML) applications in predicting geoscientific phenomena spanning from weather to climate \cite{he2021sub,molina2023review} hold the promise of enhancing the skill of deterministic \cite{love2009real,toms2019testing,silini2021machine,suematsu2022machine,martin2022using,hagos2022observationally} and probabilistic \cite{delaunay2022interpretable} forecast of the MJO. Improvement in the forecast skill has been achieved also by applying ML techniques for correcting the forecasts of dynamical models \cite{kim2021deep,silini2022improving}. The majority of ML models used for MJO prediction are based on artificial neural networks (ANNs). The work of \citeA{delaunay2022interpretable} uses deep convolutional neural networks (CNNs) to quantify the uncertainty.  
 We note that the probabilistic method in \cite{delaunay2022interpretable} is not fully data driven. 
A wide array of ANN architectures has been devised for MJO prediction models. \citeA{toms2019testing}  employed two hidden layers comprising fully connected networks, while \citeA{love2009real} and \citeA{martin2022using} utilized a single hidden layer of fully connected networks. \citeA{suematsu2022machine}  employed recurrent neural networks as a form of reservoir computing, whereas \citeA{silini2022improving} employed them as autoregressive neural networks.

In terms of input variables, some of the ML models for MJO prediction utilize a selected set of atmospheric state variables, including the OLR and zonal winds, to predict one of the MJO indices \cite{toms2019testing,delaunay2022interpretable}. Others focus solely on the atmospheric state variables required for constructing  and predicting the MJO index \cite{martin2022using}. Certain models use the MJO index as both input and output \cite{love2009real,suematsu2022machine,silini2021machine,silini2022improving} or combine it with other climate indices \cite{hagos2022observationally}. Some studies suggest that increasing the number of input variables can enhance MJO forecast skill. Nonetheless, models utilizing only the MJO index as a predictor exhibit comparable forecast skill, highlighting the significance of the ML model’s characteristics. Thus, the prediction of the MJO can be regarded as a nonparametric problem, while most existing ANN models fall under parametric ML techniques. An alternative avenue for exploration lies in \textit{Gaussian processes} (GPs), which represent a nonparametric learning approach that could be harnessed for MJO prediction. The GP approach has been applied to modeling geophysical datasets such as the prediction of tide height \cite{roberts2013gp}. However, this approach is not autoregressive. 

Currently, only one of the ML models proposed for MJO forecasting offers the capability to quantify the forecast uncertainty. The model developed by \citeA{delaunay2022interpretable} predicts both the forecast mean and variance of RMM indices, providing insight into forecast reliability by using a combined model- and data-driven strategy. 
The model assumes a bivariate Gaussian distribution on the CNN \cite{lecun1995convolutional}.  The CNN is trained by maximizing the log-likelihood for each of the forecast lead times. Specifically, the CNN input is a series of daily gridded maps that include zonal wind at 200 hPa and 850 hPa, OLR, sea surface temperature, specific humidity at 400 hPa, geopotential at 850 hPa, and downward longwave radiation at the surface; and the output is the mean and variance of the forecast of RMM1 and RMM2. The output variance represents the intrinsic chaotic (aleatoric) uncertainty in the prediction. In addition, the epistemic uncertainty is estimated by using a Monte Carlo dropout method to produce an ensemble of forecasts. 
We note, however,  that this model assumes no correlation between RMM1 and RMM2 and relies only on the past day $t$ to predict the mean and variance on day $t+\tau$. It overlooks the lag correlation between RMM1 and RMM2 as outlined in \citeA{clivarmjo2009} and the potential influences of the values between day $t$ and day $t+\tau$ on the day $t+\tau$. Additionally, interpreting uncertainties derived from neural network (NN) models can be challenging because  the influence of weights $\theta$ on the NNs is not always clear and NNs may not inherently reflect probabilities. Moreover, the quality of the uncertainty estimates provided by Monte Carlo dropouts depends on choices of architecture designs, and effective design of training procedures is necessary to obtain satisfactory results \cite{verdoja2020notes}.  
Additionally, the recent short and medium range weather forecasting models such as FourCastNet \cite{pathak2022fourcastnet}, GenCast \cite{price2023gencast}, and Aardvark \cite{vaughan2024aardvark} are not amenable for forecasting MJO.

To address these gaps, we present a novel data-driven and autoregressive probabilistic model for forecasting the MJO amplitude and phase that depends only on the past MJO observations. This model harnesses the power of GPs, enabling us not only to make predictions but also to quantify the inherent uncertainties associated with these forecasts. A GP is an extension of the multivariate Gaussian distributions to infinite dimensions. In practical terms, this means that given an input vector, the process will return a probability distribution of the observation vector based on the input. As a result, GPs provide a natural way to quantify uncertainty in predictions. GPs offer greater interpretability and transparency compared to NNs. This clarity stems from the GP's covraiance kernel, which provides more readily understandable insights into the model behavior than the complex array of parameters found in NNs \cite{stein1999interpolation, myren2021comparison}. As statistical models, GPs provide insight into how predictions are made, and the covariance function of a GP reveals the relationships among input features and their impact on predictions. Furthermore, GPs typically involve fewer hyperparameters to tune when compared with NNs, leading to increased computational efficiency.

Specifically, the contributions of this paper are as follows:
\begin{itemize}
    \item Introduction of a probabilistic framework for the MJO based on GP models that are trained using empirical correlations to improve forecast accuracy.
    \item Development of a nonparametric strategy utilizing GP models to directly provide uncertainty levels in MJO forecasts that do not rely on ensemble prediction.
    \item Proposal of a posteriori covariance correction extending probabilistic MJO coverage over three weeks.
    \item Enhancement of interpretability and transparency compared to neural network models, alongside improved computational efficiency due to fewer hyperparameters.
\end{itemize}

The paper is organized as follows. In Section \ref{sec:meth} we present the data utilized in this study and describe our methodology for forecasting the MJO. In Section \ref{sec:metrics} we elaborate on the metrics used for analyzing the performance of the proposed model compared to observations and dynamical forecast systems. Section \ref{sec:res} showcases the results we have obtained in this work. In Section \ref{sec:conc} we discuss our findings and present directions for future work.

\section{Methodology}
\label{sec:meth}
\subsection{Data}
\label{subsec:data}
The daily MJO RMM index dataset \footnote{ \url{http://www.bom.gov.au/climate/mjo/graphics/rmm.74toRealtime.txt}} used in the study is provided by the Bureau of Meteorology. RMM1 and RMM2 values are available from June 1, 1974, to the most recent date. Because of missing values before 1979, we select the January 1, 1979, to December 31, 2023, range for our study. %
The dataset is divided into three subsets: \textit{i}) the training set used to determine the parameters of the prediction and corresponding variance, January, 1, 1979 to December 31, 2016; \textit{ii}) the validation set used to obtain the corrected variance with increasing lags, January 1, 2007 to December 31, 2011; and \textit{iii}) the test set used to verify the results, January 1, 2012 to December 31, 2023. The start of predictions in the validation set $(t_v=\text{Jan--01--2007})$ and test set $(t_0=\text{Jan--01--2012})$ are part of the model input. The training dataset is further divided into $n=10,000$ samples of length $L={40,60}$ days.

\subsection{GP model}%
\label{subsec:alg}

In this work we obtain the probability distribution of predicted RMM indices. The entire algorithm for our method is described in the diagram shown in Figure \ref{fig:gp-mjo alg}. The details related to time series prediction and Gaussian process are provided in \ref{sec:back} and the mathematical framework of the proposed method in \ref{sec:alg}.
\begin{figure}[!htb]
    \centering
    \includegraphics[scale=0.38]{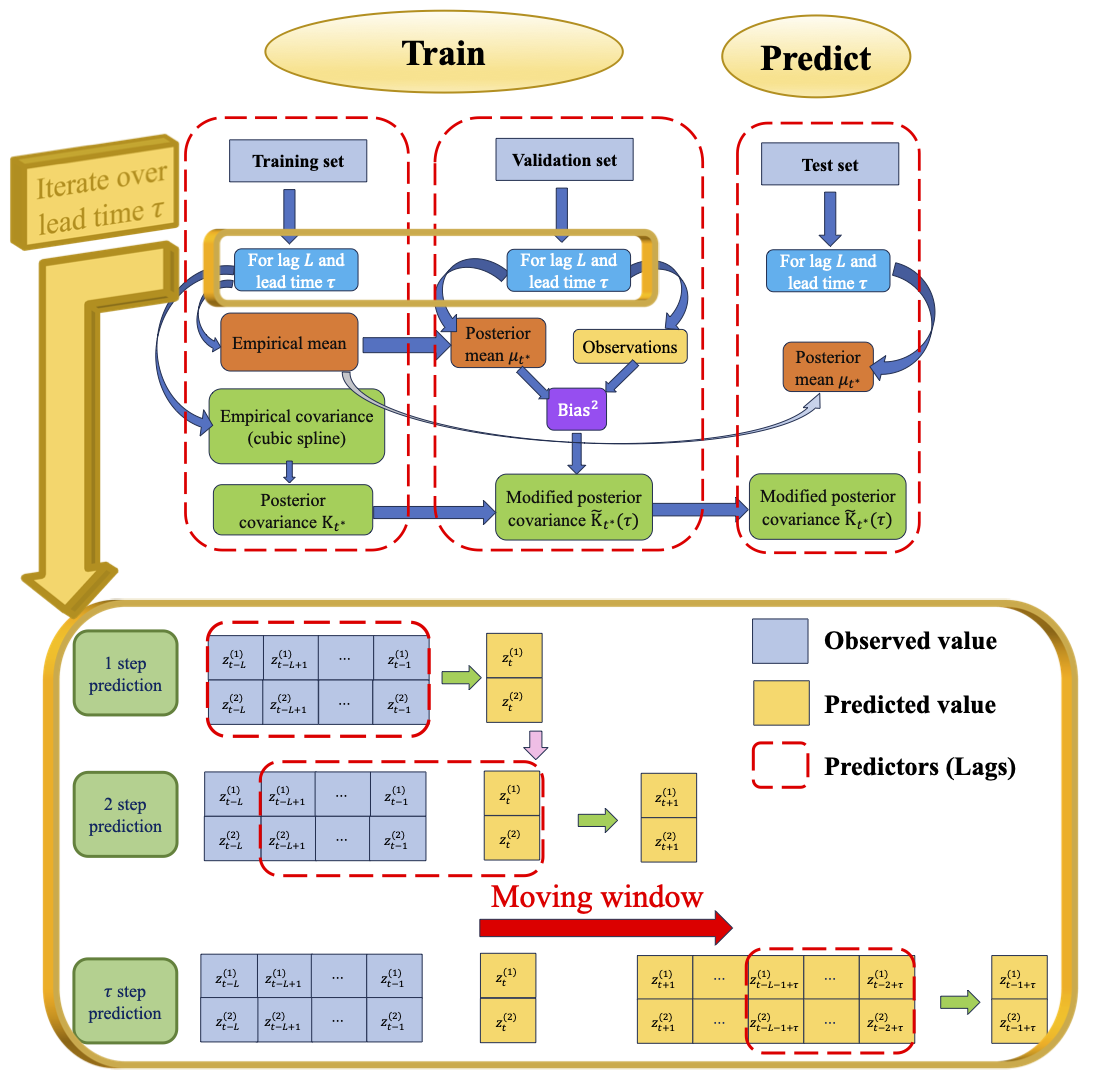}
    \caption{Flowchart of the entire algorithm. \textit{Top:} Diagram of the GP model for the MJO forecast. The blue arrows indicate the order of operations in the algorithm. $t^*$ represents the predicted timestamp, $\text{Bias}^2$ is the square of the bias between the predicted values and the true observations. \textit{Bottom:} Iterated method for the multistep time series forecasting for two outputs with lag = $L$, lead time = $\tau$ ($\tau>L$). $z_t^{(1)}$, $z_t^{(2)}$ are the values of RMM1 and RMM2 at time $t$. The green arrows indicate one-day-ahead predictions. The red arrows indicate the moving window of the predictors. Including the predictions from the previous step as predictors in the current step is indicated by the pink arrow. See \ref{sec:alg} for more details.}
    \label{fig:gp-mjo alg}
\end{figure}

We denote the values of RMM1 and RMM2 on the $t$th day by $z_t^{(1)}$ and $z_t^{(2)}$, respectively. As shown by the diagram in Fig. \ref{fig:gp-mjo alg}, the input to the GP model is a contiguous time series of RMM1 and RMM2 of length $L$ (blue rectangles). $L$ is referred to as lag in days and corresponds to $T_1=T_2=L$ in \ref{sec:back}. The goal of this work is to obtain the predictive distribution of the vector $[z_t^{(1)}, z_t^{(2)}]^\top$ (yellow rectangles) at the next $\tau$ times conditioned on the previous $L$ days:
\begin{linenomath*}
\begin{equation}
    p\Big(
    \left.
    \begin{bmatrix}
        z_{L+1:L+\tau}^{(1)}\\
        z_{L+1:L+\tau}^{(2)}
    \end{bmatrix}
    \right| \begin{bmatrix}
        z_{1:L}^{(1)}\\
        z_{1:L}^{(2)}
    \end{bmatrix}; \Theta \Big)\,,
\end{equation}
\end{linenomath*}
where $\Theta $ is the parameter of the distribution. 
We will model $[ z_t^{(1)}, z_t^{(2)}]^\top$ as a bivariate GP.

The model employs a classical regression algorithm based on one-step-ahead Gaussian process predictions. The one-step-ahead approach involves making predictions at step $k$ using all available information up to step $k-1$. This information is assumed to be Gaussian (normal) distributed. A Gaussian process (GP) is a collection of random variables, such that any finite set of which has multivariate Gaussian distribution \cite{williams2006gaussian}. A GP is specified by two functions: the mean function $\mu(\cdot)$ and the covariance function $K(\cdot, \cdot)$. The mean function represents the expected value of the process at any given time. It provides a baseline prediction and captures the trend of the timeseries. The covariance function, also known as the kernel, describes how points in the time series are related to each other. It captures the periodicity and other patterns in the data as well as the uncertainties in the time series. Using the GP model, the time series of RMM1 and RMM2 can be modeled as:
\begin{linenomath*}
    \begin{equation}
        f(Z) \sim \mathcal{N} ( \mu( Z ), 
     K(Z, Z') ),
    \end{equation}
\end{linenomath*}
where $Z = \begin{bmatrix}
    z_t^{(1)}\\
    z_t^{(2)}
\end{bmatrix} 
 =\begin{bmatrix}
    \text{RMM1}(t)\\
    \text{RMM2}(t)
\end{bmatrix}$, and $Z' = \begin{bmatrix}
    z_{t'}^{(1)}\\
    z_{t'}^{(2)}
\end{bmatrix} 
 =\begin{bmatrix}
    \text{RMM1}(t')\\
    \text{RMM2}(t')
\end{bmatrix}$, $f(Z)$ is the bivariate time series of RMMs, where $t,t'$ represent all the time indexes in the series.

During the training, the model takes as input $n$ overlapping batches of RMM1 and RMM2 indices, each of length $L$. The training data is then divided into an input subset $\bfX^{(1:2)}=[\bfX^{(1)};\bfX^{(2)}]$ and an output subset $\bfy^{(1:2)}=[\bfy^{(1)};\bfy^{(2)}]$, each of length $2L$. These subsets are used to estimate an empirical mean by the average of the corresponding subsets. The empirical covariance function is estimated by partitioning the training data into four blocks that represent the covariance between all inputs $\Cov[\bfX^{(1:2)},\bfX^{(1:2)}]$, covariance between all outputs $\Cov[\bfy^{(1:2)},\bfy^{(1:2)}]$, cross-covariance between inputs and outputs $\Cov[\bfX^{(1:2)},\bfy^{(1:2)}]$, and cross-covariance between outputs and inputs $\Cov[\bfy^{(1:2)},\bfX^{(1:2)}]$. The cross- and auto-covariance of the RMMs is modeled using a cubic spline interpolation of the cross- and auto-correlations of the indices, shown in Figure \ref{fig:entire corrs of rmms}.

During the validation, the empirical mean and covariance are used to predict the posterior mean $\bfmu_{t^*}$ and covariance $\bfK_{t^*}$ at time $t^*$. The details of these calculations are provided in the Appendix \ref{subsec:empirical GPs}. As the one-step-ahead prediction is iterated forward, the last prediction becomes input for the next prediction (the red dashed rectangle). Therefore, when predictions are carried out into the future, ``observations'' are replaced by the predictions. As the prediction window moves farther ahead of the start time, more and more components of the input vectors are replaced by GP predictions. This process introduces systematic uncertainties because the covariance is related only to the lag value $L$ and not to the lead time $\tau$ of the prediction or the predictor values. At leads beyond $L$ the predictive variance should increase because of the uncertainties introduced by replacing observations with predicted values. The covariance function must be corrected to account for the additional uncertainty. We design the correction by computing the average variance bias between the posterior mean and true observations. This bias is then added to the covariance function at each forecast lead time to obtain the modified posterior covariance $\Tilde{\bfK}_{t^*}(\tau)$. The details of these calculations are provided in the Appendix \ref{subsec:cov update}.

One important element of the GP model is the confidence interval of the forecast, which is the confidence region of the normal distribution characterized by the posterior mean and corrected covariance function. \citeA{johnson2002applied} have shown that $(1-\alpha)$ confidence region of the $p$-variate (or multivariate) normal distribution is a hyperellipsoid bounded by chi-square distribution with $p$ degrees of freedom at the level $\alpha$. Since RMMs are bivariate time series, here $p=2$ in our GP model. Therefore, the ellipsoid of the $(1-\alpha)$ confidence region for the GP model is centered on the posterior mean with the axes $\pm \chi_{2}(\alpha) \sqrt{\lambda_i} \bfe_i$, $i=1, 2$, where $\{\lambda_i\}_{i=1}^{2}$ and $\{\bfe_i\}_{i=1}^2$ are the eigenvalues and eigenvectors of the corrected covariance $\Tilde{\bfK}_{t^*}(\tau)$. 

A limitation of this confidence interval estimation is that it relies on normality assumptions; nevertheless, due to its relatively smooth behavior, it is a reasonable assumption, which is also supported by our numerical results. 

\begin{figure}
    \centering
    \includegraphics[scale=0.225]{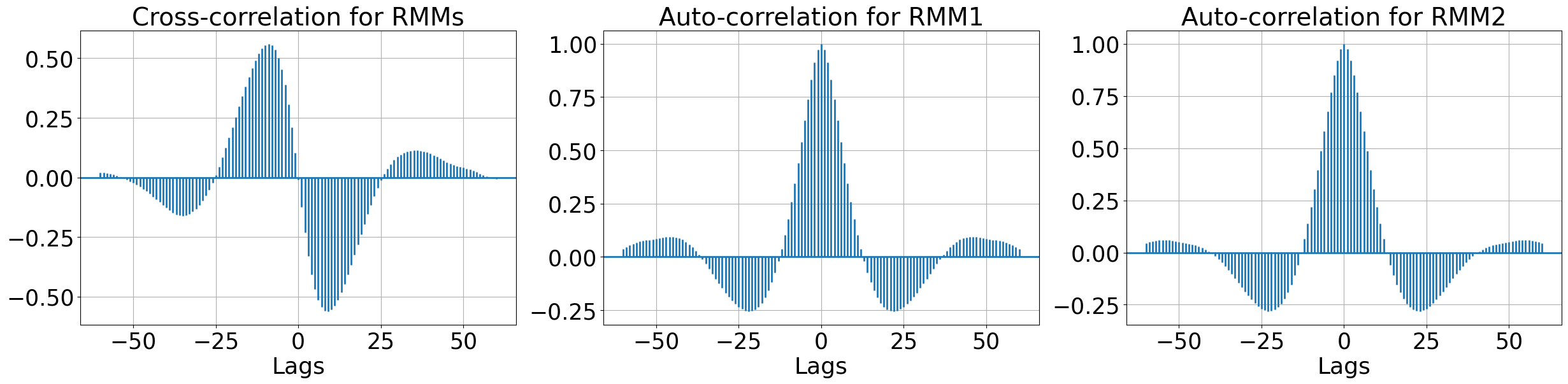}
    \caption{Cross-correlations and auto-correlations of RMMs with maximum lag = 60 days.}\vspace{-0.5cm}
    \label{fig:entire corrs of rmms}
\end{figure}

\section{Metrics}
\label{sec:metrics}
We will use two different types of quantitative metrics to analyze the performance of our models.  

\subsection{Deterministic prediction skill}
For the deterministic prediction skill, we use the predictive mean of the GP model, obtained from equation (\ref{eq:emp-conditional-mean}), as the RMM predictions, denoted by $(\hatz_t^{(1)}, \hatz_t^{(2)})$ in the subsequent equations.  
The performance of the model is measured by the bivariate correlation coefficient (COR) and root mean squared error (RMSE) defined as follows:
\begin{linenomath*}
\begin{equation}\label{eq:cor}
    \text{COR}(\tau) = \frac{\sum_{t=1}^{n_p} \big( z_{t}^{(1)} \hatz_{t}^{(1)}(\tau) + z_{t}^{(2)} \hatz_{t}^{(2)}(\tau) \big) }{ \sqrt{ \sum_{t=1}^{n_p} \Big( \big( z_t^{(1)} \big)^2 + \big( z_t^{(2)} \big)^2 \Big) } \sqrt{ \sum_{t=1}^{n_p} \Big( \big( \hatz_{t}^{(1)}(\tau) \big)^2 + \big( \hatz_{t}^{(2)}(\tau) \big)^2 } \Big) },   
\end{equation}
\end{linenomath*}
\begin{linenomath*}
\begin{equation}\label{eq:rmse}
    \text{RMSE}(\tau) = \sqrt{\frac{1}{n_p} \sum_{t=1}^{n_p} \Big( \big( z_{t}^{(1)} - \hatz_{t}^{(1)}(\tau) \big)^2 + \big( z_{t}^{(2)} - \hatz_{t}^{(2)}(\tau) \big)^2 \Big) },
\end{equation}
\end{linenomath*}
where $z_{t}^{(1)}$ and $z_{t}^{(2)}$ are the observations of RMM1 and RMM2 on the $t$th day in the test set, $\hatz_{t}^{(1)}(\tau)$ and $\hatz_{t}^{(2)}(\tau)$ are the predictions of RMM1 and RMM2 on the $t$th day in the test set  for the lead time of $\tau$ days, and $n_p$ is the number of the predictions.

We also analyze the phase error $E_{\phi}$ and the amplitude error $E_{A}$ of RMMs defined as
\begin{linenomath*}
\begin{equation}\label{eq:phase error}
    E_{\phi}(\tau) = \frac{1}{n_p} \sum_{t=1}^{n_p} \tan^{-1}
    \Big(
    \frac{ 
    z_t^{(1)} \hatz_t^{(2)}(\tau) - z_t^{(2)} \hatz_t^{(1)}(\tau) 
    }{
    z_t^{(1)}\hatz_t^{(1)}(\tau) + z_t^{(2)}\hatz_t^{(2)}(\tau)
    }
    \Big),
\end{equation}
\end{linenomath*}
\begin{linenomath*}
\begin{equation}\label{eq:amplitude error}
    E_{A} (\tau) = \frac{1}{n_p} \sum_{t=1}^{n_p} \big( \hatA_{t}(\tau) - A_{t} \big).
\end{equation}
\end{linenomath*}
The phase error \cite{rashid2011prediction} $E_{\phi}$ is the angle in degrees ($-180^{\circ} - 180^{\circ}$), which represents the averaged angle between the observed RMM vector $(z_t^{(1)}, z_t^{(2)})$, and the predicted RMM vector $(\hatz_t^{(1)}, \hatz_t^{(2)})$. $A_{t}$ is the observation of RMM amplitude on the $t$th day in the test set, and $\hatA_{t}(\tau) \colon = \sqrt{ \big(\hatz_{t}^{(1)}(\tau) \big)^2 + \big( \hatz_{t}^{(2)}(\tau) \big)^2 }$ is the predicted amplitude on the $t$th day in the test set for the lead time of $\tau$ days. The evaluation is conducted for two values of the lag, $L=40,60$, size of the training set $n = 10000$, size of the validation set $n_v=2000$, number of predictions for computing the errors $n_p=528$, and forecast lead time $\tau=1,2,\ldots,60$.

To better visualize the skill of the model for the MJO phase, we also assess the model’s skill by the Heidke skill score (HSS) \cite{heidke1926berechnung} defined in equation (\ref{eq:has for mjo}). 

HSS is a measure of how well a forecast is relative to a randomly selected forecast. HSS plays a crucial role in evaluating the accuracy of deterministic forecasts.  The definition of HSS \cite{hyvarinen2014probabilistic} is given by
\begin{linenomath*}
\begin{equation}\label{eq:hss}
    \mathrm{HSS} = \frac{\mathrm{PC}-E}{1-E} = \frac{2(ad-bc)}{(a+b)(b+d)+(a+c)(c+d)},
\end{equation}
\end{linenomath*}
where $a$, $b$, $c$, $d$ are different numbers of cases observed to occur in each category in the contingency table (see Table \ref{tab:hss-contingency}); PC is the proportion correct defined as
\begin{linenomath*}
\begin{equation}\label{eq:hss-pc}
    \mathrm{PC}=\frac{a+d}{a+b+c+d};
\end{equation}
\end{linenomath*}
$E$ is the expectation of the probability of the correct forecasts defined as
\begin{linenomath*}
\begin{equation}\label{eq:hss-expect}
    E = p( \{z_{t} \in \calA,\hatz_{t} \in \calA \} \cup \{z_{t} \notin \calA ,\hatz_{t} \notin \calA \} )
    = p(z_{t} \in \calA ) p(\hatz_{t} \in \calA ) + p(z_{t} \notin \calA) p(\hatz_{t} \notin \calA);
\end{equation}
\end{linenomath*}
and its maximum-likelihood estimate is given by
\begin{linenomath*}
\begin{equation}\label{eq:hss-expect-mle}
    E = \Big( \frac{a+c}{a+b+c+d} \Big) \Big( \frac{a+b}{a+b+c+d} \Big) + \Big( \frac{b+d}{a+b+c+d} \Big) \Big( \frac{c+d}{a+b+c+d} \Big).
\end{equation}
\end{linenomath*}
\begin{table}[h]%
    \centering
    \begin{tabular}{|c|c|c|c|}
        \hline
        \multicolumn{2}{|c|}{\multirow{2}{*}{\# of cases}} & \multicolumn{2}{c|}{Observation $z_{t} \in \calA$}\\
        \cline{3-4}
        \multicolumn{2}{|c|}{} & True & False \\
        \hline
        \multirow{2}{*}{Forecast $\hatz_{t} \in \calA$} & True & $a$ (true positive/hit) & $b$ (false positive/false alarm) \\
        \cline{2-4}
        & False & $c$ (false negative/miss) & $d$ (true negative/correct rejection) \\
        \hline
    \end{tabular}
    \vspace{2mm}
    \caption{Contingency table}
    \label{tab:hss-contingency}
\end{table}
To combine the strong/weak MJO and 8 phases, we divide the plane into 9 parts and introduce phase 0 (inactive MJO) by defining $\{ \calA_i \}_{i=0}^{8}$ as follows:
\begin{linenomath*}
\begin{equation}\label{eq:phase 0 for hss}
   ( z_t^{(1)}, z_t^{(2)} ) \in \calA_0 \iff \sqrt{(z_t^{(1)})^2 + (z_t^{(2)})^2} < 1,
\end{equation}
\end{linenomath*}
\begin{linenomath*}
\begin{equation}\label{eq:phase 1-8 for hss}
    ( z_t^{(1)}, z_t^{(2)} ) \in \calA_i\; (i=1,\ldots,8) \iff
    \begin{cases}
        & \mathrm{atan2}(z_t^{(2)},z_t^{(1)})  \in (-\pi,-\frac{3}{4}\pi] + \frac{\pi}{4} (i-1)  \\
        & \quad\mathrm{and}\quad \sqrt{(z_t^{(1)})^2 + (z_t^{(2)})^2} \geq 1,
    \end{cases}
\end{equation}
\end{linenomath*}
where $(z_t^{(1)}$, $z_t^{(2)})$ are the observations of (RMM1, RMM2) at time $t$ and atan2 is the 2-argument arctangent function whose range is $(-\pi,\pi]$. For the strong/weak MJO ($i=0$) and each MJO phase $i$ ($i=1,\ldots,8$), we can calculate the corresponding HSS$(i)$ by setting $\calA \colon =\calA_i$ in equations \eqref{eq:phase 0 for hss} and \eqref{eq:phase 1-8 for hss} and applying them to $\calA$ in Table \ref{tab:hss-contingency}. Hence,
\begin{linenomath*}
\begin{equation}\label{eq:has for mjo}
    \mathrm{HSS}(i) = \frac{2(a_i d_i - b_i c_i)}{(a_i + b_i)(b_i + d_i) + (a_i + c_i)(c_i + d_i)},
\end{equation}
\end{linenomath*}
where $a_i=\mathbf{card}\left. \Big(t  \right| ( z_t^{(1)}, z_t^{(2)} ) \in \calA_i$ and $( \hatz_t^{(1)}, \hatz_t^{(2)} ) \in \calA_i \Big)$; $b_i=\mathbf{card}\left. \Big(t  \right| ( z_t^{(1)}, z_t^{(2)} ) \notin \calA_i$ and $( \hatz_t^{(1)}, \hatz_t^{(2)} ) \in \calA_i \Big)$; $c_i=\mathbf{card}\left. \Big(t  \right| ( z_t^{(1)}, z_t^{(2)} ) \in \calA_i$ and $( \hatz_t^{(1)}, \hatz_t^{(2)} ) \notin \calA_i \Big)$; $d_i=\mathbf{card}\left. \Big(t  \right| ( z_t^{(1)}, z_t^{(2)} ) \notin \calA_i$ and $( \hatz_t^{(1)}, \hatz_t^{(2)} ) \notin \calA_i \Big)$, $i=0,1,\ldots,8$; $(z_t^{(1)}$, $z_t^{(2)})$ are the observations of (RMM1, RMM2) at time $t$; and $(\hatz_t^{(1)}$, $\hatz_t^{(2)})$ are the predictions of (RMM1, RMM2) at time $t$. Note that $\mathbf{card}\left(\cdot \right)$ denotes the cardinality of the set, which is the number of elements in the set. In our case, it represents the
number of days $t$ where the corresponding condition is met.

\subsection{Probabilistic prediction skill}
The probabilistic nature of the GP model allows a natural evaluation of the probabilistic skill of the MJO prediction. We assess the model using two probabilistic scores: continuous ranked probability score (CRPS) \cite{hersbach2000decomposition} and the log score \cite{roulston2002evaluating}.

CRPS is a scoring rule that compares a single ground truth value to a cumulative distribution function, first introduced in \cite{matheson1976scoring} and widely used in weather forecasts. It is defined as 
\begin{linenomath*}
    \begin{equation}
        \text{CRPS}(F_D,y) = \int_{\bbR} \Big( F_D(x) - H(x \geq y) \Big)^2 \,dx,
    \end{equation}
\end{linenomath*}
where $F_D$ is the cumulative distribution function of the forecasted distribution $D$, $H$ is the Heaviside step function and $y\in\bbR$ is the observation. We assume the forecasted distribution $D$ is Gaussian distribution, then the CRPS formula is given by
\begin{linenomath*}
    \begin{equation}
    \label{eq:crps gaussian}
        \text{CRPS}(\calN(\mu, \sigma^2), y) = \sigma \left( 
        \omega (2 \Phi(\omega)-1) + 2\phi(\omega) - \frac{1}{\sqrt{\pi}}
        \right), \quad
        \omega = \frac{y-\mu}{\sigma},
    \end{equation}
\end{linenomath*}
where $\Phi(\cdot)$ and $\phi(\cdot)$ are cumulative distribution function and probability density functions of the standard normal distribution $\calN(0,1)$. The CRPS for MJO is then computed as the sum of the CRPS for RMM1 and RMM2 following \cite{marshall2016visualizing}.

The negative log-likelihood of the normal distribution is used to compute the log score, which is given as follows:
\begin{linenomath*}
    \begin{equation}
    \label{eq:log-like}
        \text{Log-Score}(\tau)
        = \frac{1}{n_p} \sum_{t=1}^{n_p} 
        \frac{1}{2} \left(  
        \log(2\pi) + 
        \log \vert \bfSigma_{t} (\tau) \vert + 
        \begin{bmatrix}
            z_t^{(1)} - \hatz_t^{(1)}(\tau)\\ 
            z_t^{(2)} - \hatz_t^{(2)}(\tau)
        \end{bmatrix}^{\top} 
        \bfSigma_t(\tau)^{-1} 
        \begin{bmatrix}
            z_t^{(1)} - \hatz_t^{(1)}(\tau)\\ 
            z_t^{(2)} - \hatz_t^{(2)}(\tau)
        \end{bmatrix}
        \right),
    \end{equation}
\end{linenomath*}
where $\bfSigma_t(\tau) \in \bbR^{2 \times 2}$ is the covariance matrix of the predictions of RMM1 and RMM2 on the $t$th day for the lead time of $\tau$ days, and $\vert \bfSigma_{t} (\tau) \vert$ is the determinant of the covariance matrix $\bfSigma_{t} (\tau)$.

\section{Results}
\label{sec:res}
In this section we present the results of the prediction skill of our model in Section \ref{subsec:pred skill}, the results of HSS for each MJO phase over the forecast lead time in Section \ref{subsec:hss}, and the visualizations of the uncertainty quantification with the GP model in Section \ref{subsec:uq}.

\subsection{Prediction skill}
\label{subsec:pred skill}

\begin{figure}[!htb]
    \centering
    \includegraphics[width=\textwidth]{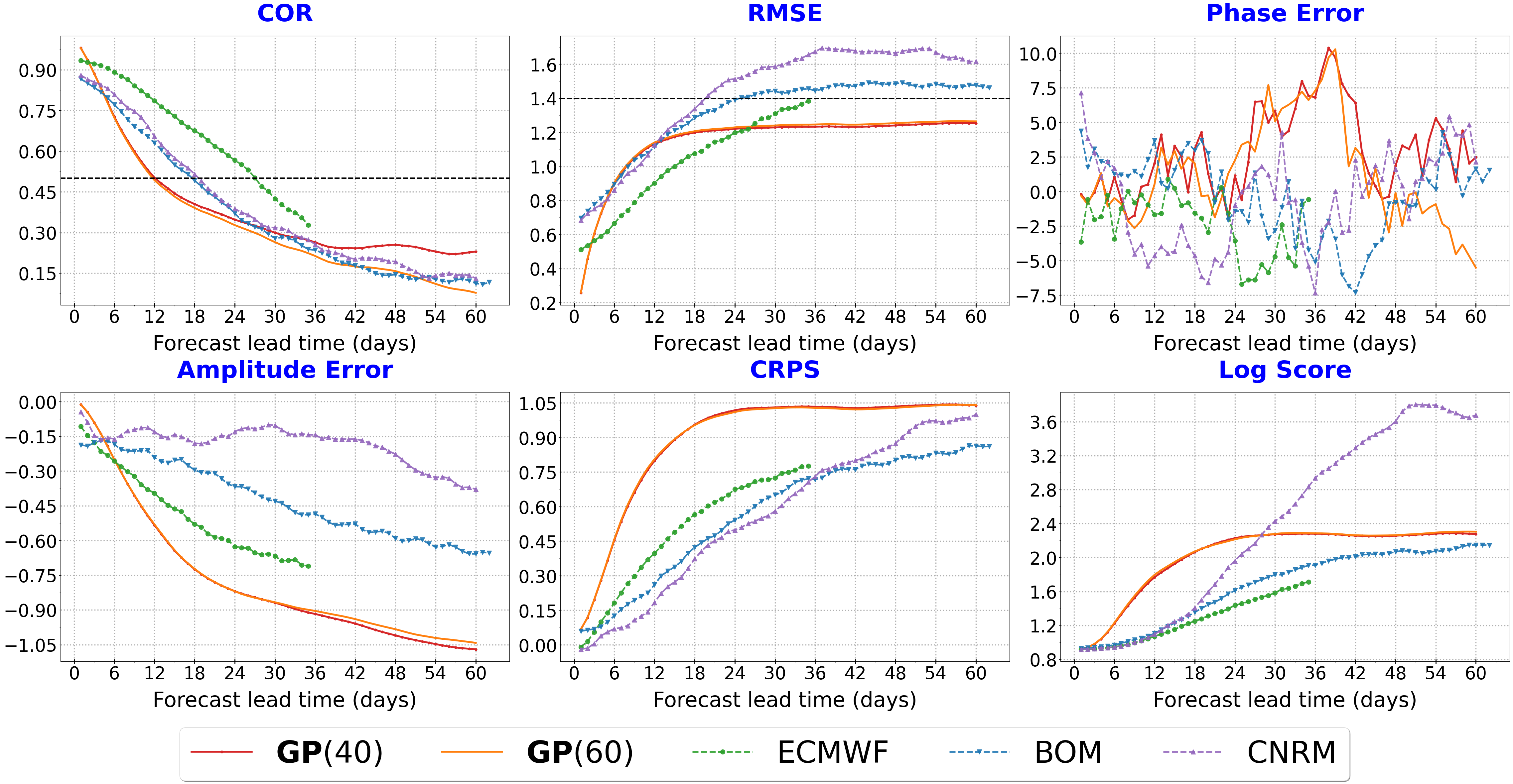}
    \caption{Prediction skill quantifiers and errors of the GP model with lag $L=40$, $60$, respectively, compared to three models in the sub-seasonal to seasonal prediction project (S2S). \textit{Top:} COR (higher is better), RMSE (lower is better), and phase error (degress) (lower is better) over 528 predictions. \textit{Bottom:} Amplitude error (lower absolute is better), CRPS (lower is better), and log score (negative log-likelihood) (lower is better) over 528 predictions. Red lines and orange lines represent the GP model with lag $L=40$ and $L=60$ respectively, green lines represent the European Center for Medium-Range Weather Forecasts (ECMWF), blue lines represent the Bureau of Meteorology (BOM), purple lines represent the Centre National de Recherche M\'et\'eorologiques (CNRM).}
    \label{fig:entire errors}
\end{figure}

Figure \ref{fig:entire errors} presents the results of the prediction skill and errors of the GP model compared to the sub-seasonal to seasonal prediction project (S2S) models, including the European Center for Medium-Range Weather Forecasts (ECMWF) with 35 forecast lead days, Bureau of Meteorology (BOM) with 62 lead days, and Centre National de Recherche M\'et\'eorologiques (CNRM) with 60 lead days. The metrics are calculated from predictions made on different days for each model, as the S2S models are initialized on different dates. We calculated the metrics for the GP model and ECMWF over the same period from January 3, 2012, to January 10, 2017, and for the BOM and CNRM models over the same period from January 1, 1993, to December 15, 2014. The values COR $=0.5$ and RMSE $=1.4$ are the commonly used skill thresholds for a climatological forecast \cite{rashid2011prediction}. In this figure we see that our model has a prediction skill of 12 days for both lag $L=40$ and $L=60$ with threshold COR $=0.5$. The ECMWF model demonstrates the best overall performance for COR. While the GP model performs best during the first three forecast lead days, it declines rapidly and eventually reaches similar COR values as the BOM and CNRM models. Regarding the RMSE, the prediction skill is longer than 60 days for $L=40$ and $L=60$ with threshold RMSE $=1.4$. 
The GP model has a much lower RMSE than S2S models during the first three forecast days, then RMSE increases to values larger than in ECMWF over the next 20 lead days. It eventually stabilizes around RMSE $=1.25$, outperforming BOM and CNRM across the full 60 forecast lead days. The fast decline of COR for the GP model is due to the fact that we use the empirical correlations from historical RMMs of large size in our model. Specifically, when the forecast lead time increases, the predicted RMMs will become smaller and smoother because of the empirical correlations over a long period of time, giving rise to the smaller variations of RMMs than the true observations and therefore a lower COR. The small value of the predicted RMMs also accounts for the tiny changes in RMSE after day 24 of the forecast lead time. As for the phase error (the angle of RMMs in degrees), %
we observe that most phase errors for the GP model are positive and larger than ECMWF and CNRM, indicating a faster propagation relative to the observations. 
For the amplitude errors, we note that all of them are negative. Because of the smaller values of the predicted RMMs of the GP model with forecast lead time increasing, the amplitude is underestimated, resulting in negative and worse amplitude errors than S2S models. The GP model performs worse than the S2S models in terms of probabilistic skill, as measured by CRPS and the log score (negative log-likelihood). This is due to the larger variances in the GP model, causing its probability distribution to diverge significantly from the observations. We also note that the results with lags $L=40$ and $L=60$ are similar; therefore, for the rest of the paper we will show only results with lags $L=40$.

\subsection{HSS}
\label{subsec:hss}
\begin{figure}[!htb]
    \centering
    \includegraphics[scale=0.35]{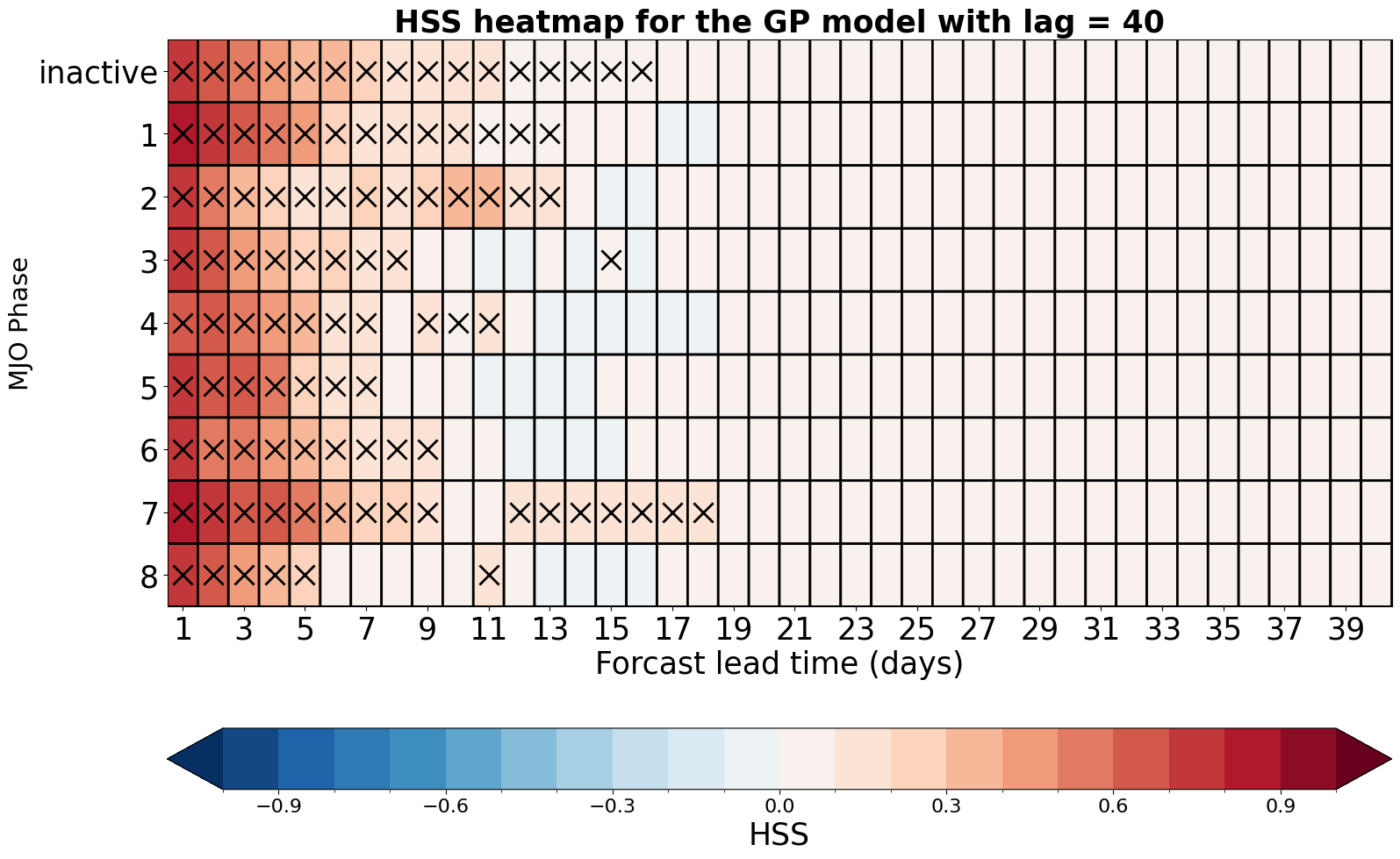}
    \caption{HSS heatmap for the GP model over 528 predictions with lag $L = 40$ (higher HSS is better). The cells with black cross marker ``X'' represent the significant samples from \textit{Fisher's exact test} with the critical value $\alpha=0.05$.}
    \label{fig:hss}
\end{figure}
Figure \ref{fig:hss} shows the HSS heatmap for the combination of phases (1--8) and inactive (weak) MJO for the forecast lead times (1--40 days) over 528 predictions. From this figure we can see that our model has a positive skill for most phases and forecast lead times and has high skill scores for the first 10 forecast lead days for all 8 phases and inactive MJO. We also use Fisher's exact test \cite{fisher1922interpretation} with critical value $\alpha=0.05$ to determine the significant samples for HSS. The cells with the black cross marker in Figure 1 indicate the statistically significant associations between observations and forecasts, which is consistent with the results of Section \ref{subsec:pred skill} indicating that the model has a good prediction skill within the first 12 days of the forecast lead time. The results reported above provide better skill than the ANN model results reported by \cite{kim2021deep} for the first five forecast lead days in terms of correlation coefficient and overall in terms of root mean square error.

\subsection{Uncertainty quantification}
\label{subsec:uq}
Here we pick two samples (Nov--03--2012 to Jan--01--2013, Jan--14--2013 to Mar--14--2013) out of $n_p=528$ predictions with $\tau=60$ forecast lead days to present the uncertainty quantification of the predicted MJO.
We compare the GP model with the ECMWF ensemble means, including standard deviations from 11 members, which performs best among the S2S models, as well as with observations from BOM. Figure \ref{fig:entire ts nov-jan} shows an example in which the MJO is mostly inactive within 60 days, and Figure \ref{fig:entire ts jan-mar} shows an example of an active MJO event. These two examples show that predictions of the GP model capture the general trend seen in observations and outperforms ECMWF during the first 5 lead days. The $\pm \sigma$ confidence intervals (CI) grow as the forecast lead time increases and cover a larger portion of the observation range compared to the ECMWF model's CI. To obtain the complete picture of MJO prediction, we summarize results in Figure \ref{fig:entire mjo phase diagram example}, which shows the MJO phase diagram for Nov--03--2012 to Jan--01--2013 and Jan--14--2013 to Mar--14--2013 of our model with $68.0\%$ confidence region. The figure clearly shows that almost all observations (black lines) mostly lie within the confidence region (colorful shadings), which demonstrates the quality of the uncertainty quantification of our model. Animated phase diagrams can also be found on the project website \url{https://gp-mjo.github.io/}, which show how the elliptical confidence region enlarges with time.
\begin{figure}[!htb]
    \centering
    \includegraphics[width=\textwidth]{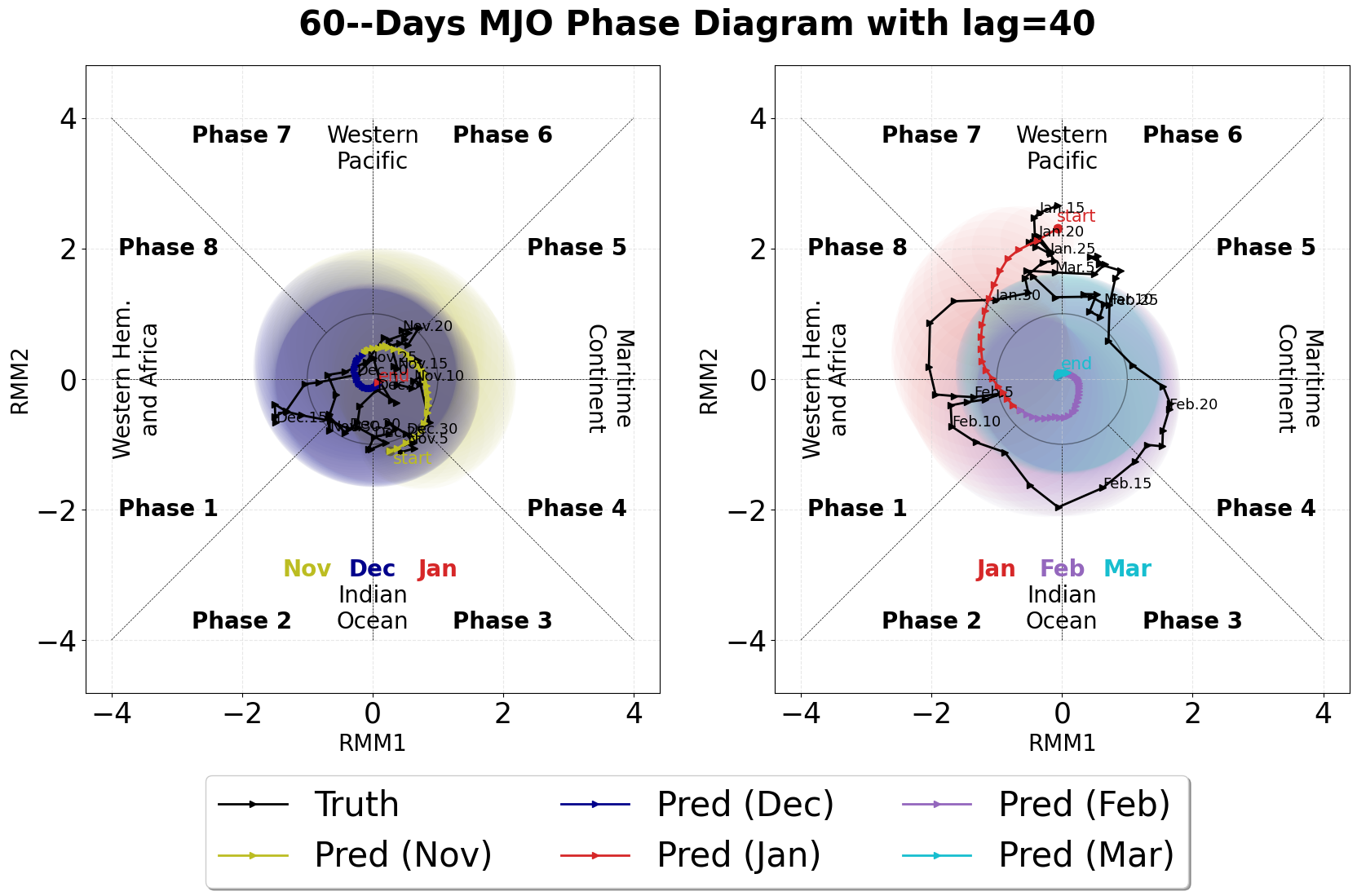}
    \caption{\textit{Left:} 60-days MJO phase diagram for Nov--03--2012 to Jan--01--2013 with lag $L = 40$. Black lines are observations (truth). Olive lines are predictions in November, and olive shadings are 68\% confidence regions (CR) in November.  
    Dark blue lines are predictions in December, and dark blue shadings are CR in December. Red lines are predictions in January, and red shadings are CR in January.  \textit{Right:} 60-days MJO phase diagram for Jan--14--2013 to Mar--14--2013 with lag $L = 40$. Black lines are observations (truth). Red lines are predictions in January, and red shadings are CR in January. Purple lines are predictions in February, and purple shadings are CI in February. Cyan lines are predictions in March, and cyan shadings are CR in March.}
    \label{fig:entire mjo phase diagram example}
\end{figure}

\begin{figure}[ht]%
    \centering
    \includegraphics[width=\textwidth]{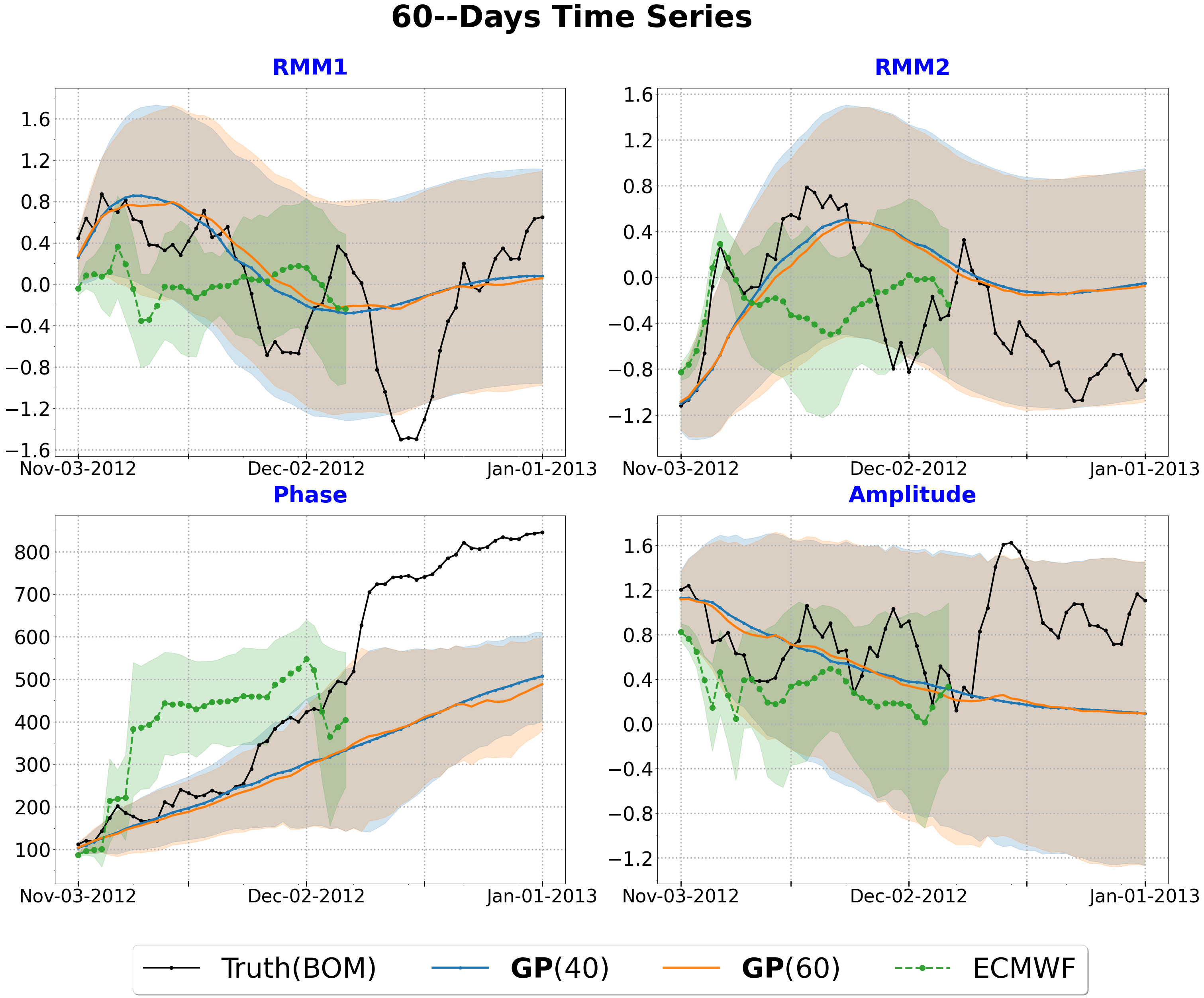}
    \caption{60-days time series of MJO for Nov--03--2012 to Jan--01--2013 for lag $L=40,60$. We denote observations (truth) from the BOM by black dots; predictions of the GP model for lag $L=40$ and $L=60$ by blue cross and orange cross, respectively; $\pm \sigma$ CI of the GP model  
    for lag $L=40$ and $L=60$ by blue shading and orange shading, respectively; predictions of the ECMWF model by green dots, $\pm \sigma$ CI of the ECMWF model by green shading. \textit{Top left}: Time series of RMM1. \textit{Top right}: Time series of RMM2. \textit{Bottom left}: Time series of phase (angle in the degrees on looping). \textit{Bottom right}: Time series of amplitude.}
    \label{fig:entire ts nov-jan}
\end{figure}

\begin{figure}[ht]%
    \centering
    \includegraphics[width=\textwidth]{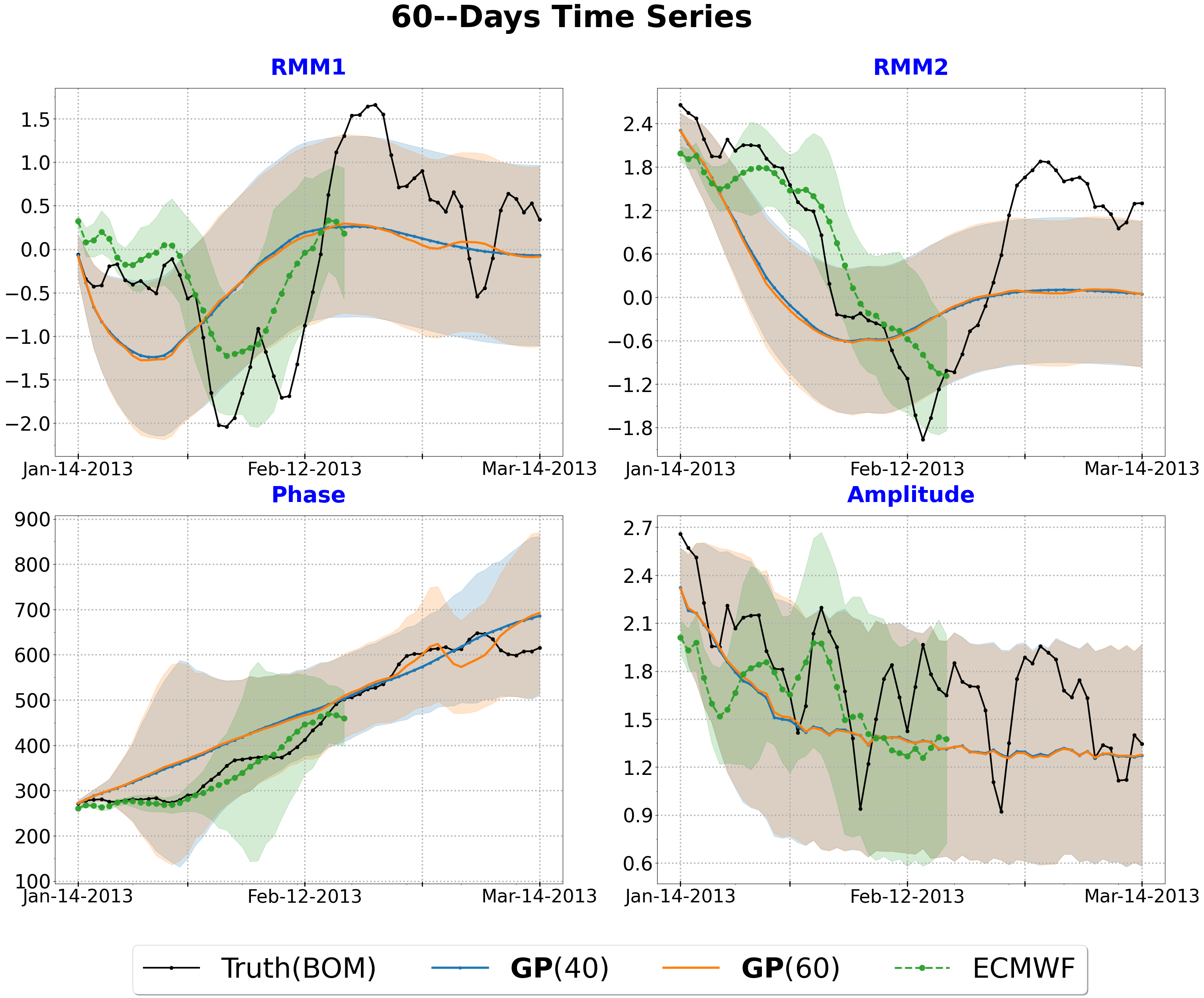}
    \caption{60-days time series of MJO for Jan--14--2013 to Mar--14--2013 for lag $L=40,60$. We denote observations (truth) from the BOM by black dots; predictions of the GP model for lag $L=40$ and $L=60$ by blue cross and orange cross, respectively; $\pm \sigma$ CI of the GP model  
    for lag $L=40$ and $L=60$ by blue shading and orange shading, respectively; predictions of the ECMWF model by green dots, $\pm \sigma$ CI of the ECMWF model by green shading. \textit{Top left}: Time series of RMM1. \textit{Top right}: Time series of RMM2. \textit{Bottom left}: Time series of phase (angle in the degrees on looping). \textit{Bottom right}: Time series of amplitude.}
    \label{fig:entire ts jan-mar}
\end{figure}

\section{Conclusions}
\label{sec:conc}
In this study we have developed a robust, probabilistic, data-driven model to predict the MJO with high accuracy and quantify prediction uncertainty using GPs with empirical correlations. Our methodology primarily focused on employing the daily RMM index dataset from January 1, 1979, to December 31, 2023, to train, test, and validate the model. We have successfully demonstrated that our model's mean prediction of the daily RMM index remains accurate within a 12-day forecast window, as evidenced by our evaluations using metrics including the correlation, RMSE, phase errors, amplitude errors, CRPS, log score, and the HSS.

The specific aspect that provides the model's efficacy lies in the approach used to handle GPs for time series prediction and uncertainty quantification. We avoid the typical need for optimizing hyperparameters, thus streamlining the process and enhancing the model's efficiency and stability. This approach is driven by using training data to empirically determine covariance, which is then fitted to a continuous function. The advantage of this method is twofold. It offsets the need for external hyperparameters and ensures stability, especially for long-term predictions, where the model reverts to the mean or prior. Furthermore, our model is robust to the lags of predictors, maintaining accuracy and reliability in predictions without being significantly impacted by lag beyond a certain threshold. This characteristic is especially notable in the context of long-term forecasting and in scenarios where data input may be subject to variable delays.

Moreover, our prediction also provides uncertainty bounds. The uncertainty in our method is state-independent, meaning it is unrelated to the initialized MJO event and depends solely on lead time. The probabilistic model's confidence region covers the observations well, maintaining an average coverage of close to 60 days. This aspect is crucial for reliable forecasting in dynamic and uncertain climatic conditions governed by the MJO. Assuming that the dynamic model fit through a Gaussian process is optimal, this study indeed suggests that the limit of predictability of RMM1 and RMM2 based on their history alone is constrained to the results presented in this paper. Furthermore, it indicates that the memory of the dynamical system, based on these inputs, is limited to about 40 to 60 days in the past. 

The approach proposed in this study can be improved by including aspects of seasonal variability and adding additional predictors. In our future work we aim to mitigate these limitations by incorporating seasonal factors into the model and expanding the range of physical variables in the inputs. These aspects are expected to improve our GP model performance significantly. Additionally, while effective, our current empirical approach to constructing GPs could be further advanced by exploring parametric methods in modeling GPs. This future direction could potentially offer more nuanced insights and greater precision in our predictions.

In summary, this study introduces a new data-driven method for predicting the MJO, providing a reliable, efficient, and robust model that provides competitive accuracy and offers extensive insight into prediction uncertainties. As we move forward, our focus will be on refining and enhancing this model to address its current limitations and adapt it to the challenges in climatic forecasting.

\appendix
\section{Background}
\label{sec:back}
In this section we review the probabilistic forecasting and the iterative method for the time series forecasting in Section \ref{subsec:iterative meth} and GP models in Section \ref{subsec:gp}.

\subsection{Probabilistic forecasting with an iterative method}
\label{subsec:iterative meth}
In the general probabilistic forecasting problem \cite{rangapuram2018deep,wang2019deep}, we usually denote $M$ univariate time series by $\{ z_{1:T_j}^{(j)} \}_{j=1}^{M}$, where $z_{1:T_j}^{(j)} \colon = (z_{1}^{(j)}, z_{2}^{(j)}, \ldots, z_{T_j}^{(j)})$ is the $j$th time series and $z_{t}^{(j)}$ is the value of the $j$th time series at time $t$, $1 \leq t\leq T_j$. Our goal is to model the distribution of $z_{T_j+1:T_j+\tau}^{(j)}$ at the next $\tau$ time conditioned on the past:
\begin{linenomath*}
\begin{equation}
    p(\left. z_{T_j+1:T_j+\tau}^{(j)} \right| z_{1:T_j}^{(j)}; \Theta), \quad j=1,\ldots,M,
\end{equation}
\end{linenomath*}
where $\Theta$ is the set of the learnable parameters shared by all $M$ time series.

The objective of multistep time series forecasting \cite{weigend2018time,cheng2006multistep,sorjamaa2007methodology} is to predict $M$-variate time series at the next $\tau$ time $\{ z_{T_j+1:T_j+\tau}^{(j)} \}_{j=1}^{M}$ given $\{ z_{1:T_j}^{(j)} \}_{j=1}^M$, where $\tau>1$. A multistep prediction is typically carried out using the iterative method. In this technique, the values computed for each step ahead are sent to the next step as inputs. The iterative method can be written in the autoregressive model as follows:
\begin{linenomath*}
\begin{equation}\label{eq:ar}
    \begin{bmatrix}
        z_t^{(1)}\\
        \vdots\\
        z_t^{(M)}
    \end{bmatrix}
    = \begin{bmatrix}
        f_1(z_{t-T_1:t-1}^{(1)})\\
        \vdots\\
        f_M(z_{t-T_M:t-1}^{(M)})
    \end{bmatrix},
\end{equation}
\end{linenomath*}
where %
$f_1,\ldots,f_M$ are random functions. After the learning process, the predicted values at the next $\tau$ time are given by
\begin{linenomath*}
\begin{equation}
    \hatz_{t+\tau-1}^{(j)}=
    \begin{cases}
        f_j(z_{t-T_j:t-1}^{(j)}) & \text{if $\tau=1$}\\
        f_j(z_{t-T_j-1+\tau:t-1}^{(j)},\hatz_{t:t-2+\tau}^{(j)}) & \text{if $\tau=2,\ldots,T_j$}\\
        f_j(\hatz_{t-T_j-1+\tau:t-2+\tau}^{(j)}) & \text{if $\tau=T_j+1,\ldots$},
    \end{cases}
\end{equation}
\end{linenomath*}
where $j=1,\ldots,M$, $\hatz_t^{(j)}$ is the predicted value of the $j$th sequence of time series at time $t$.  The lower diagram in Figure \ref{fig:gp-mjo alg} illustrates the case where $M=2$, $T_1=T_2=L$ for the iterated method. The iterated method has also been applied to many classical machine learning models such as \textit{recurrent neural networks} \cite{medsker2001recurrent, galvan2001multi,yunpeng2017multi} and \textit{hidden Markov models}  \cite{rabiner1986introduction,rossi2006volatility,horelu2015forecasting}.

\subsection{Gaussian processes}
\label{subsec:gp}
A Gaussian process  \cite{williams2006gaussian} is a collection of random variables such that every finite number of which has a multivariate normal distribution. A GP is defined by a mean function $\mu(\cdot)$ and a covariance function $K(\cdot,\cdot)$ and is denoted by $\mathcal{GP}(\mu(\cdot),K(\cdot,\cdot))$.

Given a dataset $\calD=\{ \bfX,\bfy \}$ comprising the inputs $\bfX=\{ \bfx_i \}_{i=1}^n$ (where $\bfx_i \in \bbR^d$) and the corresponding observations $\bfy = (y_1,y_2,\ldots,y_n)^\top$ (where $y_i \in \bbR$), suppose $y_i = f(\bfx_i)$, where $f\colon \bbR^d \rightarrow \bbR$ is a random function. %
Gaussian process regression assumes that the unknown function is a prior GP, denoted as $f(\cdot) \sim \mathcal{GP}( \mu(\cdot), K(\cdot,\cdot) )$. Then the posterior distribution at a set of test points $\bfX^{*}=\{ \bfx_{i}^{*} \}_{i=1}^m$ (where $\bfx_i^* \in \bbR^d$) has the following form:
\begin{linenomath*}
\begin{equation}\label{eq:conditional-distribute}
    p(\left. f(\bfX^*) \right| 
    \calD ) = \calN( \bbE[f(\bfX^*) \vert \calD] , \Cov[f(\bfX^*) \vert \calD] ),
\end{equation}
\end{linenomath*}
with the posterior mean and covariance as follows:
\begin{linenomath*}
\begin{subequations}\label{eq:conditional-meanvar}
\begin{align}
    \bbE[f(\bfX^*) \vert \calD] &= \mu( \bfX^* ) + K(\bfX^*,\bfX) \big[K(\bfX,\bfX) \big] ^{-1}\left(\bfy - \mu(\bfX) \right),\label{eq:conditional-mean}\\
    \Cov[f(\bfX^*) \vert \calD] &= K(\bfX^*, \bfX^*) - K(\bfX^*,\bfX)
    \big[K(\bfX,\bfX) \big]^{-1} K(\bfX,\bfX^*). \label{eq:conditional-variance}
\end{align}
\end{subequations}
\end{linenomath*}

\section{Algorithm}
\label{sec:alg}
\subsection{Empirical GPs for the bivariate time series}
\label{subsec:empirical GPs}
Here we denote the bivariate time series of RMMs by $\{z_t^{(j)}\}_{t=1}^T$, $j=1,2, \cdots, T$, where $T$ is the length of the entire time series. As before we assume that we model the two time series by a joint GP:
\begin{linenomath*}
\begin{equation}\label{eq:emp-gp for bivariate ts}
    \begin{bmatrix}
    z_t^{(1)}\\
    z_t^{(2)}
\end{bmatrix} \sim \mathcal{GP} \Big(\mu\big( \begin{bmatrix}
    z_t^{(1)}\\
    z_t^{(2)}
    \end{bmatrix}\big), 
    K \big(  \begin{bmatrix}
    z_t^{(1)}\\
    z_t^{(2)}
    \end{bmatrix}, 
    \begin{bmatrix}
    z_{t'}^{(1)}\\
    z_{t'}^{(2)}
    \end{bmatrix}  \big)
    \Big).
\end{equation}
\end{linenomath*}
We seek to calculate the distribution of the two components at the next time step conditioned on the previous $L$ values. In other words, we need to calculate the  predictive distribution of $[z_t^{(1)}, z_t^{(2)}]^\top$ at time $t^*$ for the lag $L$, which is expressed as
\begin{linenomath*}
\begin{equation}\label{eq:emp-conditional-distribute}
    p \Big( \left. \begin{bmatrix}
        z_{t^*}^{(1)}\\
        z_{t^*}^{(2)}
    \end{bmatrix} \right|
    \begin{bmatrix}
        z_{t^*-L:t^*-1}^{(1)}\\
        z_{t^*-L:t^*-1}^{(2)}
    \end{bmatrix}
     \Big) 
    = \calN ( 
    \bfmu_{t^*}
    , 
   \bfK_{t^*}
    ),.
\end{equation}
\end{linenomath*}
The predictive mean and covariance, $\bfmu_{t^*} \in \bbR^{2 \times 1}$, $\bfK_{t^*} \in \bbR^{2 \times 2}$, are estimated by following (\ref{eq:emp-conditional-mean}) and (\ref{eq:emp-conditional-variance}):
\begin{linenomath*}
\begin{align}\label{eq:emp-conditional-mean}
    \bfmu_{t^*} =&\; \bbE \Big[\left. 
    \begin{bmatrix}z_{t^*}^{(1)}\\
   z_{t^*}^{(2)}
    \end{bmatrix} \right|
    \begin{bmatrix}
        z_{t^*-L:t^*-1}^{(1)}\\
        z_{t^*-L:t^*-1}^{(2)}
    \end{bmatrix}
     \Big]\nonumber\\ 
      =&\; \bbE \Big[ \begin{bmatrix}
        z_{t^*}^{(1)}\\
        z_{t^*}^{(2)}
    \end{bmatrix} \Big]
    + \Cov \Big[  \begin{bmatrix}
        z_{t^*}^{(1)}\\
        z_{t^*}^{(2)}
    \end{bmatrix},
    \begin{bmatrix}
        z_{t^*-L:t^*-1}^{(1)}\\
        z_{t^*-L:t^*-1}^{(2)}
    \end{bmatrix}
     \Big]\nonumber\\
     &\; \Cov \Big[  \begin{bmatrix}
        z_{t^*-L:t^*-1}^{(1)}\\
        z_{t^*-L:t^*-1}^{(2)}
    \end{bmatrix},
    \begin{bmatrix}
        z_{t^*-L:t^*-1}^{(1)}\\
        z_{t^*-L:t^*-1}^{(2)}
    \end{bmatrix}
     \Big] ^{-1}
    \left( 
    \begin{bmatrix}
        z_{t^*-L:t^*-1}^{(1)}\\
        z_{t^*-L:t^*-1}^{(2)}
    \end{bmatrix}
    -\bbE\Big[ \begin{bmatrix}
        z_{t^*-L:t^*-1}^{(1)}\\
        z_{t^*-L:t^*-1}^{(2)}
    \end{bmatrix} \Big] 
    \right)\nonumber\\
    \approx &\; \bbE\Big[ \begin{bmatrix}
        \bfy^{(1)}\\
        \bfy^{(2)}
    \end{bmatrix} \Big] 
    + \Cov \Big[  \begin{bmatrix}
        \bfy^{(1)}\\
        \bfy^{(2)}
    \end{bmatrix},
    \begin{bmatrix}
        \bfX^{(1)}\\
        \bfX^{(2)}
    \end{bmatrix}
     \Big]\nonumber\\
    &\; \Cov \Big[  \begin{bmatrix}
        \bfX^{(1)}\\
        \bfX^{(2)}
    \end{bmatrix},
    \begin{bmatrix}
        \bfX^{(1)}\\
        \bfX^{(2)}
    \end{bmatrix}
     \Big] ^{-1}
    \left( 
    \begin{bmatrix}
        z_{t^*-L:t^*-1}^{(1)}\\
        z_{t^*-L:t^*-1}^{(2)}
    \end{bmatrix}
    -\bbE\Big[ \begin{bmatrix}
        \bfX^{(1)}\\
        \bfX^{(2)}
    \end{bmatrix} \Big] 
    \right),
\end{align}
\end{linenomath*}
\begin{linenomath*}
\begin{align}\label{eq:emp-conditional-variance}
\bfK_{t^*} = &\; \Cov \Big[  \left. \begin{bmatrix}
        z_{t^*}^{(1)}\\
        z_{t^*}^{(2)}
    \end{bmatrix} \right|
    \begin{bmatrix}
        z_{t^*-L:t^*-1}^{(1)}\\
        z_{t^*-L:t^*-1}^{(2)}
    \end{bmatrix}
     \Big]\nonumber\\
    = &\; \Cov \Big[   \begin{bmatrix}
        z_{t^*}^{(1)}\\
        z_{t^*}^{(2)}
    \end{bmatrix},
    \begin{bmatrix}
        z_{t^*}^{(1)}\\
        z_{t^*}^{(2)}
    \end{bmatrix}
     \Big]
    - \Cov \Big[   \begin{bmatrix}
        z_{t^*}^{(1)}\\
        z_{t^*}^{(2)}
    \end{bmatrix},
    \begin{bmatrix}
        z_{t^*-L:t^*-1}^{(1)}\\
        z_{t^*-L:t^*-1}^{(2)}
    \end{bmatrix}
     \Big]\nonumber\\
    &\; \Cov \Big[   \begin{bmatrix}
        z_{t^*-L:t^*-1}^{(1)}\\
        z_{t^*-L:t^*-1}^{(2)}
    \end{bmatrix},
    \begin{bmatrix}
        z_{t^*-L:t^*-1}^{(1)}\\
        z_{t^*-L:t^*-1}^{(2)}
    \end{bmatrix}
     \Big]^{-1} 
     \Cov \Big[\begin{bmatrix}
        z_{t^*-L:t^*-1}^{(1)}\\
        z_{t^*-L:t^*-1}^{(2)}
    \end{bmatrix},
    \begin{bmatrix}
    z_{t^*}^{(1)}\\
    z_{t^*}^{(2)}
    \end{bmatrix}
     \Big]\nonumber\\
    \approx &\; \Cov\Big[\begin{bmatrix}
        \bfy^{(1)}\\
        \bfy^{(2)}
    \end{bmatrix}, \begin{bmatrix}
        \bfy^{(1)}\\
        \bfy^{(2)}
    \end{bmatrix} \Big]
    - \Cov \Big[\begin{bmatrix}
        \bfy^{(1)}\\
        \bfy^{(2)}
    \end{bmatrix},
    \begin{bmatrix}
        \bfX^{(1)}\\
        \bfX^{(2)}
    \end{bmatrix}
     \Big]\nonumber\\
    &\; \Cov \Big[\begin{bmatrix}
        \bfX^{(1)}\\
        \bfX^{(2)}
    \end{bmatrix},
    \begin{bmatrix}
        \bfX^{(1)}\\
        \bfX^{(2)}
    \end{bmatrix}
     \Big] ^{-1}
    \Cov \Big[\begin{bmatrix}
    \bfX^{(1)}\\
    \bfX^{(2)}
    \end{bmatrix},
    \begin{bmatrix}
    \bfy^{(1)}\\
    \bfy^{(2)}
    \end{bmatrix}
     \Big], 
\end{align}
\end{linenomath*}
where
\begin{linenomath*}
\begin{equation}\label{eq:emp-conditional-approx}
    \bfX^{(j)} =\begin{bmatrix}
        z_{1:L}^{(j)}\\
        z_{2:L+1}^{(j)}\\
        \vdots\\
        z_{n:L+n-1}^{(j)}
    \end{bmatrix}^\top \in \bbR^{L \times n},
    \quad \bfy^{(j)} = \begin{bmatrix}
        z_{L+1}^{(j)}\\
        z_{L+2}^{(j)}\\
        \vdots\\
        z_{L+n}^{(j)}
    \end{bmatrix}^\top \in \bbR^{1 \times n},
    \quad j = 1,2,
\end{equation}
\end{linenomath*}
and
\begin{linenomath*}
\begin{equation}\label{eq:construct mjo rolling window}
    \bfX^{(1:2)} \coloneqq \begin{bmatrix}
        \bfX^{(1)}\\
        \bfX^{(2)}
    \end{bmatrix} \in \bbR^{2L \times n}, \quad \bfy^{(1:2)} \coloneqq \begin{bmatrix}
        \bfy^{(1)}\\
        \bfy^{(2)}
    \end{bmatrix} \in \bbR^{2 \times n}.
\end{equation}
\end{linenomath*}
In equations (\ref{eq:emp-conditional-mean}) and (\ref{eq:emp-conditional-variance}) we use the empirical mean and covariance of $n$ batches of training data with lag $L$ to approximate the expectation of the target and the covariance of the target and predictors. 

\begin{table}[htb]
\label{alg:gp for MJO}
\begin{tabular}{ | c | p{12cm} | }

\hline
Workflow & Parameters\\
\hline

\multirow{8}{*}{%
{Input}} & $n:$ number of samples in the training dataset\\
& $L:$ number of lags\\
& $t_v:$ start index for the predictions in validation dataset\\
& $t_0:$ start index for the predictions in testing dataset\\
& $\tau:$ forecast lead time\\
& $\{ [z_t^{(1)},z_t^{(2)}] \}_{t=1}^{L+n}:$ training dataset\\
& $\{ [z_t^{(1)},z_t^{(2)}] \}_{t=t_v}^{t_v+L+\tau+n_v-2}:$ validation set\\ 
& $\{ [z_t^{(1)},z_t^{(2)}] \}_{t=t_0-L}^{t_0-1}:$ starting predictors in test set\}\\
\hline
\multirow{16}{*}{\rotatebox[origin=c]{90}
{Computation steps}} & 1. Construct the training dataset $\calD^{(1:2)}=\{ \bfX^{(1:2)}, \bfy^{(1:2)} \}$ by equations (\ref{eq:construct mjo rolling window}) and  (\ref{eq:emp-conditional-approx}), $\bfX^{(1:2)} \in \bbR^{2L \times n}$, $\bfy^{(1:2)} \in \bbR^{2 \times n}$ \\
& 2. Compute $\bbE[\bfy^{(1:2)}]$ \\
& 3. Obtain $\Cov\big[\begin{bmatrix}
    \bfX^{(1:2)}\\
    \bfy^{(1:2)}
\end{bmatrix}, \begin{bmatrix}
    \bfX^{(1:2)}\\
    \bfy^{(1:2)}
\end{bmatrix}\big]=\begin{bmatrix}
    \Cov[\bfX^{(1:2)},\bfX^{(1:2)}] & \Cov[\bfX^{(1:2)},\bfy^{(1:2)}]\\
    \Cov[\bfy^{(1:2)},\bfX^{(1:2)}] & \Cov[\bfy^{(1:2)},\bfy^{(1:2)}]
\end{bmatrix}
$ by cubic spline interpolation\\
& 4. In the validation set, obtain the $\{ \bfmu_t, \bfK_t \}_{t=t_v+L+i-1}^{t_v+L+\tau+i-2}$ condition on $\{[z_t^{(1)},z_t^{(2)}]\}_{t=t_v+t-1}^{t_v+L+i-2}$ for $i=1,\ldots,n_v$ by equations (\ref{eq:emp-conditional-mean}) and (\ref{eq:emp-conditional-variance}); here $\bfK_t$ is equivalent for all $t$ \\
& 5. In the validation set, obtain modified covariances as a function of lead time $\{ \Tilde{\bfK}_{t_v}(t-t_v+1) \}_{t=t_v}^{t_v+\tau-1}$ by (\ref{eq:bias of var}) and (\ref{eq:scale cov})\\
& 6. In the test set, obtain $\{\bfmu_t\}_{t=t_0}^{t_0+\tau-1}$ by equation (\ref{eq:emp-conditional-mean}) \\
& 7. In the test set, apply the covariances obtained in the validation set to the covariances in the test set according to the corresponding lead time, $\Tilde{\bfK}_{t_0}(l) \gets  \Tilde{\bfK}_{t_v}(l)$, $l=1,\ldots,\tau$\\
& 8. Return $\bfmu_{t}$, $\Tilde{\bfK}_{t_0}(t-t_0+1)$, $t=t_0, \ldots, t_0+\tau-1$\\

\hline

\multirow{2}{*}{%
{Output}} & $\{ \bfmu_{t} \}_{t=t_0}^{t_0+\tau-1}:$ predicted mean of $\{ [\hatz_{t}^{(1)},\hatz_{t}^{(2)}] \}_{t=t_0}^{t_0+\tau-1}$\\
& $ \{ \Tilde{\bfK}_{t_0}(t-t_0+1) \}_{t=t_0}^{t_0+\tau-1}:$ predicted covariance of $\{ [\hatz_{t}^{(1)},\hatz_{t}^{(2)}] \}_{t=t_0}^{t_0+\tau-1}$\\
\hline
\end{tabular}
\vspace{0.3em}
\caption{GP model for the MJO forecast}
\end{table}

\subsection{Covariance update}
\label{subsec:cov update}
The forecast lead time is reached by repeated one-step predictions. Therefore, the covariance $\bfK_{t^*}$ in equation (\ref{eq:emp-conditional-variance}) is related only to the value of lag $L$, which is 40 or 60 in our study 
and is unrelated to the lead time $\tau$ or the predictor values. However, as we predict for longer lead times, the predictive variance should increase because of the uncertainties introduced by replacing observations by predicted values. To account for this additional uncertainty, we propose the following covariance correction. For each lead time we use a validation set of size $n_v(L)$  
with lag $L$ 
to compute the averaged variance bias between the posterior mean and true observations. Hence, the corrected variance $\Tilde{V}_*^{(j)}(\tau)$ is given by
\begin{linenomath*}
\begin{equation}\label{eq:bias of var}
\begin{aligned}
    \Tilde{V}_*^{(j)}(\tau) \coloneqq \Var[z_{t^*}^{(j)}(\tau)] &\approx \Var[ \hatz_{t^*}^{(j)}(\tau) ] + \mathrm{Bias}\Big( \hatz_{t^*}^{(j)}(\tau), z_{t^*}^{(j)}(\tau) \Big)^2,\\
    &\approx \bfK_{t^*}[j,j] + \frac{1}{n_v}\sum_{t=1}^{n_v} \big(\hatz_{t}^{(j)}(\tau) - z_{t}^{(j)}(\tau) \big)^2,
\end{aligned}
\end{equation}
\end{linenomath*}
where $\hatz_{t^*}^{(j)}(\tau)$ is the predicted value for lead time $\tau$ obtained by the above iteration, $z_{t^*}^{(j)}(\tau)$ is the corresponding true observation, and $\bfK_{t^*}[j,j]$ is the $[j,j]$th entry of the covariance matrix $\bfK_{t^*}$, $j=1,2$. Then we scale the $\bfK_{t^*}$ to the corrected covariance $\Tilde{\bfK}_{t^*}(\tau)$ for lead time $\tau$ in (\ref{eq:scale cov}) by using the variances $\{ \Tilde{V}_*^{(j)}(\tau) \}_{j=1}^2$.
Therefore, the corrected covariance $\Tilde{\bfK}_{t^*}(\tau)$ corresponds to the lead time $\tau$ and can be scaled via the following transformation:
\begin{linenomath*}
\begin{equation}\label{eq:scale cov}
    \bfK_{t^*} = \begin{bmatrix}
        \bfK_{t^*}[1,1] & \bfK_{t^*}[1,2]\\
        \bfK_{t^*}[2,1] & \bfK_{t^*}[2,2]
    \end{bmatrix}
    \longrightarrow
    \Tilde{\bfK}_{t^*}(\tau)=
    \begin{bmatrix}
        \Tilde{V}_*^{(1)}(\tau) & \frac{ \bfK_{t^*}[1,2]
        \sqrt{\Tilde{V}_*^{(1)}(\tau)} \sqrt{\Tilde{V}_*^{(2)}(\tau)}
         }{ \sqrt{\bfK_{t^*}[1,1]} \sqrt{\bfK_{t^*}[2,2]} }\\
        \frac{ \bfK_{t^*}[2,1] \sqrt{\Tilde{V}_*^{(1)}(\tau)} \sqrt{\Tilde{V}_*^{(2)}(\tau)}
        }{ \sqrt{\bfK_{t^*}[1,1]} \sqrt{\bfK_{t^*}[2,2]} } & \Tilde{V}_*^{(2)}(\tau),
    \end{bmatrix}
\end{equation}
\end{linenomath*}
where $\Tilde{V}_*^{(1)}(\tau)$ and $\Tilde{V}_*^{(2)}(\tau)$ are defined in equation \eqref{eq:bias of var}. This corrected covariance is ultimately used to estimate the confidence region described below.

\subsection{Estimation of the confidence region}
\label{subsec:confidence region}
To obtain the confidence region of the distribution $\calN(\bfmu_{t^*}, \Tilde{\bfK}_{t^*}(\tau))$, we first introduce Lemmas \ref{lem:conf region} and  \ref{lem:hyperellipsoid} as follows.
\begin{lemma}\label{lem:conf region}
    \textnormal{(Result 4.7 in Section 4.2 in \cite{johnson2002applied})}
    Let $\calN_p(\bfmu,\bfSigma)$ denote a $p$-variate normal distribution with location $\bfmu$ and known covariance $\bfSigma$. Let $\bfx \sim \calN_p(\bfmu,\bfSigma)$. Then
    \begin{enumerate}[label=(\alph*)]
    \item $(\bfx-\bfmu)^\top \bfSigma^{-1} (\bfx-\bfmu)$ is distributed as $\chi_p^2$, where $\chi_p^2$ denotes the chi-square distribution with $p$ degrees of freedom.
    \item The $\calN_p(\bfmu,\bfSigma)$ distribution assigns probability $1-\alpha$ to the solid hyperellipsoid $\{ \bfx \colon (\bfx-\bfmu)^\top \bfSigma^{-1} (\bfx-\bfmu) \leq \chi_p^2(\alpha)  \}$, where $\chi_p^2(\alpha)$ denotes the upper $(100\alpha)$th percentile of the $\chi_p^2$ distribution.
    \end{enumerate}
\end{lemma}
\begin{proof}
    See proof of Result 4.7 in Section 4.2 in \cite{johnson2002applied}.
\end{proof}
\begin{lemma}\label{lem:hyperellipsoid}
    \textnormal{((4-7) in Section 4.2 in \cite{johnson2002applied})} 
    The hyperellipsoids $\{ \bfx: (\bfx - \bfmu)^\top \bfSigma^{-1} (\bfx-\bfmu)  =c^2 \}$ are centered at $\bfmu$ and have axes $\pm c \sqrt{\lambda_i}\, \bfe_i $, where $\lambda_i$'s, $\bfe_i$'s are the eigenvalues and eigenvectors of $\bfSigma$, namely, $\bfSigma \bfe_i = \lambda_i \bfe_i$, $i=1,2,\ldots,p$.
\end{lemma}
\begin{proof}
    From Result 4.1 in Section 4.2 in \cite{johnson2002applied} we know that if $\bfSigma$ is positive definite and $\bfSigma\bfe_i = \lambda_i \bfe_i$, then $\lambda_i>0$ and $\bfSigma^{-1} \bfe_i = \frac{1}{\lambda_i} \bfe_i$. That is, $(\frac{1}{\lambda_i},\bfe_i )$ is an eigenvalue-eigenvector pair for $\bfSigma^{-1}$. According to the definition of the hyperellipsoid in quadratic form, we can conclude that the hyperellipsoids $\{ \bfx: (\bfx - \bfmu)^\top \bfSigma^{-1} (\bfx-\bfmu)  =c^2 \}$ are centered at $\bfmu$ and have axes $\pm c \sqrt{\lambda_i}\, \bfe_i $.
\end{proof}
According to the above lemmas, the $(1-\alpha)$ confidence region of the $p$-variate normal distribution is a hyperellipsoid bounded by $\chi_p^2(\alpha)$, the chi-square distribution with $p$ degrees of freedom at the level $\alpha$ \cite{johnson2002applied}. Therefore, we can construct a confidence region for the prediction $[\hatz_{t^*}^{(1)}(\tau),\hatz_{t^*}^{(2)}(\tau)]^\top$ at lead time $\tau$, where $[\hatz_{t^*}^{(1)}(\tau),\hatz_{t^*}^{(2)}(\tau)]^\top \sim \calN(\bfmu_{t^*}, \Tilde{\bfK}_{t^*}(\tau))$ after updating the covariance.

\section*{Data Availability Statement}
The daily MJO RMM index dataset is available through the Bureau of Meteorology (\url{http://www.bom.gov.au/}) and can be accessed at \url{http://www.bom.gov.au/climate/mjo/}. The codes for the numerical experiments in this work can be found at \url{https://doi.org/10.5281/zenodo.13654353} \cite{hchen19_2024_13654353}.

\acknowledgments
This material is based upon work supported by the U.S. Department of Energy, Office of Science, Office of Advanced Scientific Computing Research (ASCR) and the SciDAC FASTMath Institute  programs under Contract No. DE-AC02-06CH11357 and Argonne National Laboratory Directed Research and Development (LDRD) program. We would also like to thank the reviewers for their constructive criticism and Hannah Christensen for her valuable help with using and interpretation of the ECMWF MJO dataset.

\bibliography{refs}

\end{document}